\documentclass[11pt]{amsart}
\usepackage{geometry,a4wide,amssymb,amsfonts,amsrefs,amsmath}                
\usepackage{setspace}
\usepackage[normalem]{ulem}

\usepackage{bbm,xcolor}
\usepackage{dsfont}
\usepackage{yfonts}
\usepackage[greek, english]{babel}
\usepackage{mathtools}
\usepackage{mathrsfs}
\usepackage{hyperref}  
\usepackage{footmisc} 
\usepackage{comment}

\usepackage{orcidlink}
\hypersetup{hidelinks}

\newcommand{\authormark}{\textsuperscript{*}}

\newcommand{\R}{\mathbb{R}}
\newcommand{\fr}{\mathfrak{R}}
\newcommand{\p}{\mathbb{P}}
\newcommand{\T}{\mathbb{T}}
\newcommand{\N}{\mathbb{N}}
\newcommand{\V}{\mathbb{V}}
\newcommand{\Q}{\mathbb{Q}}
\newcommand{\Z}{\mathbb{Z}}

\newcommand{\M}{\mathcal{M}}
\newcommand{\La}{\mathcal{L}}

\newcommand{\E}{\mathbb{E}}

\newcommand{\e}{\mathcal{E}}
\newcommand{\I}{\mathcal{I}}

\newcommand{\cpy}{\mathtt{cap}}
\newcommand{\dom}{\mathtt{dom \;}}
\newcommand{\bp}{{\mathbf{p}}}
\newcommand{\bq}{{\mathbf{q}}}

\theoremstyle{plain}

\newtheorem{theorem}{Theorem}[section]
\newtheorem{lemma}[theorem]{Lemma}
\newtheorem{proposition}[theorem]{Proposition}
\newtheorem{corollary}[theorem]{Corollary}

\newtheorem{example}[theorem]{Example}
\newtheorem{remark}[theorem]{Remark}

\title[On perturbations preserving the connectivity properties in tree percolations]
{On perturbations that preserve the connectivity properties in tree percolations}

\author[M. Maxmud]{Mirmuhsin Maxmud\authormark\,\raisebox{-0.1ex}{{\large\orcidlink{0000-0002-7608-6504}}}}

\address{Research Unit of Mathematical Sciences, The University of Oulu, Finland}
\email{Mirmukhsin.Makhmudov@oulu.fi}

\thanks{\authormark\ In earlier publications, this author's name also appeared as Mirmukhsin Makhmudov.}

\author[V. Suomala]{Ville Suomala}

\address{Research Unit of Mathematical Sciences, The University of Oulu, Finland}
\email{Ville.Suomala@oulu.fi}


\thanks{\hspace*{10.4pt}Both authors were supported by the Research Council of Finland via the project “\emph{Fractals and randomness}”, grant no. 368817.}

\begin{document}

\begin{abstract}
   We consider a general bond percolation on an infinite locally finite tree, where the edge retention probabilities $p_e$ are replaced by $\min\{1,q_{|e|}p_e\}$, where  $\{q_n\}_{n\ge 1}$ is a sequence of positive perturbation factors and $|e|$ denotes the distance between the edge $e$ and the root. 
   If the original percolation model admits infinite clusters, it is of interest to investigate under which perturbations $0<q_n\le 1$ this connectivity property is preserved. Conversely, if the original percolation does not admit infinite clusters, we are led to study the stability of such a property under perturbations satisfying $q_n\ge 1$.

   In both cases, under minimal assumptions on the original model, we show that the percolative behaviour is stable against certain quantitative non-trivial perturbations. 
   We also discuss an application of our results to the Erd\H{o}s similarity conjecture for Cantor sets.
\end{abstract}

\maketitle

\section{Introduction}

Percolation models are arguably the simplest models in statistical physics that can undergo phase transitions.
A percolation model is generally defined on any infinite graph. 
The most studied cases are the lattices such as $\Z^d$ for $d\geq 1$ and trees.
In fact, percolation models on high-dimensional lattices around their criticality point behave similarly to the ones on regular trees, as they both exhibit mean-field behaviour. Thus, they belong to the same universality class \cite{HS1990}. 

To define a (bond) percolation model with parameters $\bp=\{p_e\}_{e\in \E}\subset [0,1]$ on an infinite connected graph $\T=(\V, \E)$, one retains an edge $e\in \E$ with probability $p_e$ and removes it with probability $1-p_e$, independently of all other edges.
A fundamental question is then the existence, or lack thereof, of infinite \textit{clusters} -- infinite connected components -- in the retained random subgraph of $\T$.
Formally, this question concerns the probability of the event 
$$
\mathcal{I}:=\{\omega\subset \E: \text{the subgraph } (\V,\omega)\subset \; \T \text{ has an infinite connected component}\}
$$
on the percolation configuration space $\Omega:=\{0,1\}^\E$ under the probability measure $\p_{\bp}:=\prod_{e\in\E}(p_e \delta_1 + (1-p_e)\delta_0)$.
Another relevant cluster event, associated with a vertex $v\in\V$, is defined by
\begin{equation*}
    \I_v:=\{\omega\in\Omega: \# \mathcal{C}_v(\omega)=\infty\},
\end{equation*}
where $\mathcal{C}_v(\omega)$ denotes the cluster of $v$, defined as the connected component (subgraph) of $(\V, \omega) \subseteq \T$ containing $v$. Note that we often identify the connected subgraphs of $\T$ with their corresponding vertex set.
Then $\I=\bigcup_{v\in\V}\I_v$. 
For a connected graph $\T$ and positive edge retaining probabilities $\bp\subset(0,1]$, one can readily verify that $\p_{\bp}(\I)=1$ if and only if $\p_\bp(\I_v)>0$ for all $v\in\V$.

In this note, we consider general independent percolation models on a locally finite infinite tree $\T=(\V,\E)$, with edge retention probabilities $\bp:=\{p_e\}_{e\in\E}\subset [0,1]$. We pick one designated vertex $o$ as the root of the tree. 
Without loss of generality, we assume throughout the note that $\bp\subset(0,1]$. 
Note that for a tree $\T=(\V,\E)$, its \emph{boundary} $\partial \T$ is defined as the set of all rays, where a \emph{ray} is a simple path starting from the root that cannot be extended. 
Given $v,u\in \V$, let $[v\longleftrightarrow u]$ denote the event $u\in \mathcal{C}_v(\omega)$. If  $V\subset\V$, we let $[v\longleftrightarrow V]=\bigcup_{u\in V}[v\longleftrightarrow u]$. 
We also extend this notation for subsets of the boundary; in particular, $[v\longleftrightarrow \partial\T]$ is the event $\I_v$, provided that the tree $\T$ does not have finite rays.
Utilising the boundary, we can define a more restrictive percolation event corresponding to the \textit{\underline{uncountable (root) clusters}}: If $E\subset\partial\T$, we let
$$
[o \xleftrightarrow[]{\hspace{0.2cm} u \hspace{0.2cm}} E]:=\big\{\omega\in \{0,1\}^\E: \partial \mathcal{C}_o(\omega)\cap E\subseteq\partial \T \text{ is uncountable}\big\},
$$
where $\partial \mathcal{C}_o(\omega)$ is the boundary of the cluster $\mathcal{C}_o(\omega)$.
Note that the inclusion 
$
[o \xleftrightarrow[]{\hspace{0.2cm} u \hspace{0.2cm}} \partial\T]
\subset 
\I_o
$ 
holds trivially. It is also true that 
$
\p_\bp[o \xleftrightarrow[]{\hspace{0.2cm} u \hspace{0.2cm}} \partial\T]
=
\p_\bp(\I_o)
$
provided each individual infinite ray survives the percolation with probability zero \cite{LP-Book}*{Proposition 5.27}.
In particular, this is the case when the retention probabilities $\bp=\{p_e\}_{e\in \E}$ are \textit{homogeneous}, i.e., $p_e=p$ for all $e$ and where $0<p<1$ is fixed.
This homogeneous case represents the canonical model in the field, largely because it naturally gives rise to a critical threshold $p_c$ defined as
\begin{equation*}
    p_c(\T):=\sup\{p\in[0,1]: \p_p(\mathcal{I})=0\}=\inf\{p\in[0,1]: \p_p(\mathcal{I})=1\}\,.
\end{equation*}
The existence of infinite clusters at criticality $p_c(\T)$ depends crucially on the geometry of the tree; thus, infinite clusters at criticality may or may not exist.

This work aims to study the stability of infinite clusters in percolation models under
\textit{\textbf{level-dependent perturbations}} $\bq=\{q_n\}_{n\in\N}\subset (0,\infty)$, meaning that the perturbation applied to an edge depends only on its distance from the root.
Specifically, we ask whether the almost sure existence or absence of infinite clusters is preserved when the retention probabilities are perturbed to
\[
\bq \bp=\{\min\{q_{|e|}p_e,1\}\}_{e\in\E}\subset(0,1],
\]
where $|e|$ denotes the graph distance from the root to the endpoint of $e$ that is farther from the root. 
We note that by Hammersley’s monotone coupling \cites{Grimmett-Book-Perc, LP-Book}, if $\p_{\bq\bp}(\mathcal I)=1$, then $\p_{\bq'\bp}(\mathcal I)=1$ for any sequence $\bq'$ satisfying $q_n' \ge q_n, \; n\in\N$. This monotonicity allows us to simplify our analysis by restricting the range of the perturbation sequence based on the initial state of the system:
\begin{itemize}
\item If $\p_{\bp}(\mathcal I)=1$, we are interested in perturbations $\{q_n\}_n$ that decrease the edge retention probabilities while preserving the existence of infinite clusters. Thus, we consider perturbations $0<q_n\le 1$.
\item If $\p_{\bp}(\mathcal I)=0$, we are interested in perturbations that increase the size of percolation clusters while keeping them finite, and thus we look at perturbing sequences $\{q_n\}$ with $q_n\ge 1$.
\end{itemize} 
Heuristically, the cumulative effect of the perturbation along each infinite ray is governed by the product $\prod_{n=1}^\infty q_n$.
If this product converges to a positive and finite limit, then the perturbation merely rescales connectivity probabilities by a bounded factor (see also \eqref{relation between P_p(o and E_n) and P_{bar p}(o and E_n)}) and is, therefore, not expected to change the qualitative behaviour of the model.
The genuinely interesting perturbations are thus the \textit{divergent} ones:
\[
\prod_{n=1}^\infty q_n=0
\qquad\text{or}\qquad
\prod_{n=1}^\infty q_n=\infty,
\]
corresponding respectively to cumulative thinning and cumulative enhancement of retention probabilities.
As the following simple example demonstrates, stability under such divergent perturbations is not guaranteed in general.
\begin{example}\label{example}
Let $\T=(\V,\E)$ be a rooted binary tree. Fix a ray $\xi=(x_0=o, x_1, \dots)$ and set the edge retention probabilities to $p_e=1$ if $e=(x_{i-1}, x_{i})$ for some $i\in\N$, and $p_e=1/3$ otherwise.
While the unperturbed model $(\T,\bp)$ clearly percolates along $\xi$, it fails to percolate under any divergent thinning perturbation $\bq=\{q_n\}_{n\in\N}$, $\prod_{n=1}^\infty q_n=0$.
\end{example}
Despite such counterexamples, our main result (Theorem \ref{main theorem}) reveals that the percolative regime is rather robust. It demonstrates that under mild, necessary conditions (excluding examples such as the Example \ref{example}), the connectivity state of the system persists even when the cumulative effect of the perturbation diverges. Given a rooted tree, we denote by $V_n$ the set of vertices at distance $n$ from the root.

\begin{theorem}\label{main theorem}
Let $\T$ be an infinite locally finite rooted tree and consider a percolation on $\T$ with retention probabilities $\bp=\{p_e\}_{e\in\E}\subset (0,1]$. 
\begin{itemize}
    \item[(i)] Assume that there are no infinite clusters $\p_\bp$ almost surely. 
    If the perturbation $\bq=\{q_n\}_{n\in\N}, \: {q_n\geq 1}$, satisfies 
    \begin{equation*}
        \liminf_{n\to\infty} \prod_{i=1}^n q_i\cdot  \p_\bp[o\longleftrightarrow V_n]=0,
    \end{equation*}
    then the percolation on $\T$ with perturbed retention probabilities $\bq\bp$ does not have an infinite cluster almost surely.
    In particular, there always exists a perturbation $\bq=\{q_n\}_{n\in\N}, \; q_n>1,$ such that
    $\prod_{n=1}^\infty q_n=\infty$ and so that the perturbed model has no infinite clusters almost surely. 
    
    \item[(ii)] If 
        the percolation model $(\T, \bp)$ has an uncountable boundary cluster
    with positive probability, i.e., $\p_{\bp}[o \xleftrightarrow[]{\hspace{0.2cm} u \hspace{0.2cm}} \partial\T]>0$, then there exists a perturbation $\bq=\{q_n\}_{n\in\N}\subset (0,1)$ with $\prod_{n=1}^\infty q_n=0$, under which the almost-sure existence of an uncountable cluster is preserved.
    \\
    Furthermore, if the retention probabilities $\bp=\{p_e\}_{e\in\E}$ are \textit{homogeneous}, i.e., $p_e=p$ for all $e\in\E$, and if $\p_{\bq \bp}(\mathcal I)=1$,
    then
    \begin{equation}\label{eq: a necessary condition for the persistence in infinity clusters case}
        \sum_{n=1}^\infty \frac{1-q_n p}{p^n (\# V_n) \prod_{i=1}^n q_i}<\infty.
    \end{equation}
\end{itemize}

\end{theorem}

\begin{remark}
\begin{itemize}
    \item[(1)] In light of Example \ref{example}, the uncountable cluster condition imposed in the second part of Theorem \ref{main theorem} is also necessary.  
    
    \item[(2)] Restricted to the case of homogeneous percolation, Theorem \ref{main theorem} is only relevant if $p=p_c$. Indeed, if $p\neq p_c$ 
    the connectivity state trivially persists under any perturbation $\bq=\{q_n\}_{n\in\N}$ that does not force the system to cross the threshold, i.e., if the perturbed retention probabilities $pq_n$ remain uniformly separated from $p_c$ on the original side of the threshold.
    At criticality, Theorem \ref{main theorem} is nontrivial even for homogeneous percolation: it guarantees that the connectivity properties of the tree at criticality persist even under some divergent perturbations. 
  
    \item[(3)] The necessary condition established in \eqref{eq: a necessary condition for the persistence in infinity clusters case} is also sufficient, provided the underlying tree is \textit{spherically symmetric} (i.e., the degree of every $v \in \V$ depends solely on its distance from the root $o$) and the retention probabilities are homogeneous or, more generally, level-dependent \cite{Kersting2020}. 
    In general, however, this condition is not sufficient; a counterexample can be readily constructed using the 1-3 tree from Example 1.2 in \cite{LP-Book}.
\end{itemize}    
\end{remark}
Our proof of Theorem \ref{main theorem} relies heavily on Lyons' celebrated capacity characterisation of percolation \cites{Lyons1990, LP-Book}. 
Although some parts of the argument could be replaced by more elementary hands-on methods, the potential-theoretic characterisation of
$\p_{\bp}[o\longleftrightarrow\partial\T]$ leads to 
concise proofs and great generality. Moreover, it will be crucial in the geometric application presented in the final section,
where we apply our main theorem to a variant of the Erdős similarity conjecture for uncountable sets. 
The classical conjecture asserts that, for every infinite set $A\subset\mathbb{R}$, there exists a set of positive Lebesgue measure that contains no affine copy of $A$. The results in this note are motivated by the following variant: For every uncountable Borel set $C\subset\mathbb{R}$, there exists a full-measure set $E\subset\mathbb{R}$ into which $C$ cannot be affinely embedded. Since every uncountable Borel set contains a Cantor set, it suffices to establish this variant for all topological Cantor sets $C\subset[0,1]$. 
Partial results under various assumptions on the size or geometry of $C$ have been obtained; see Remark \ref{rmk: about the Erdos similarity conjecture} and the references \cites{JLM2025, ShmerkinYavicoli2025, IosevichYavicoli2026}. However, both the original conjecture and its uncountable variant remain wide open in full generality.
For a given Cantor set $C\subset\mathbb{R}$, the uncountable variant of the conjecture is equivalent to finding a set $B\subset\mathbb{R}$ of zero Lebesgue measure such that
\begin{equation}\label{full-measure-univ}
B\cap(\lambda C+t)\neq\varnothing
\end{equation}
for every $\lambda,t\in\mathbb{R}$ with $\lambda\neq0$. We provide an almost-sure version of this statement.
\begin{corollary}\label{cor: application of main thm to ESC}
       Given a Cantor set $C$, then there is a set $B$ of zero Lebesgue measure such that $B\cap(\lambda C+t)\neq\varnothing$ for almost all pairs $t,\lambda$.    
\end{corollary}
While the corollary as stated, also follows from recent results of Shmerkin and Yavicoli \cite{ShmerkinYavicoli2025}, the percolation approach leading to it is novel in this generality. See Proposition \ref{Erdos application} and the Remark \ref{rmk: about the Erdos similarity conjecture}.

The rest of the paper is organised as follows: In Section \ref{Lyons' capacity characterisation}, we recall Lyons' capacity characterisation of percolation, and Section \ref{proof section} is dedicated to the proof of Theorem \ref{main theorem}.
In the last section, Section \ref{Section for Applications}, we derive Corollary \ref{cor: application of main thm to ESC} using the general percolation approach.

\section{Lyons' capacity characterisation}\label{Lyons' capacity characterisation}

In this section, we introduce some necessary notations and concepts along with Lyons' capacity characterisation of percolation on arbitrary locally finite trees (including finite trees), upon which our subsequent proofs rely.

Let $\T=(\V,\E)$ be a locally finite tree. We define a partial order on the vertex set $\V$ by writing $x \rightarrow y$ if $y$ is a \textit{descendant} of $x$; that is, if $x$ lies on the unique path from the root $o$ to $y$. 
For two vertices $v_1, v_2 \in \V$, their most recent common ancestor $v_1 \wedge v_2$ is defined as the farthest vertex $v \in \V$ from the root satisfying both $v \rightarrow v_1$ and $v \rightarrow v_2$.
For any vertex $v\in\V$, let $|v|$ denote its graph distance from the root, and let $[o,v]$ denote the set of edges comprising the unique path connecting $o$ and $v$. 
Similarly, for an edge $e\in\E$, we set $|e| := |v(e)|$, where $v(e)$ denotes the endpoint of the edge $e$ farther from the root.

We extend the $\wedge$-notation to the boundary of the tree: For any two rays $\xi=(x_0=o, x_1,\dots)$ and $\eta=(y_0=o, y_1, \dots)$ in the boundary $\partial \T$, we define
\begin{equation*}
    \xi\wedge \eta :=
    \begin{cases}
        x_k, & \text{if } \xi\neq \eta, \text{ where } k = \sup\{n \geq 0 : x_i=y_i \text{ for all } 0 \leq i \leq n\};\\
        \xi, & \text{if } \xi=\eta.
    \end{cases}
\end{equation*}
We can then endow $\partial\T$ with a metric $\mathtt{d}$ by setting $\mathtt{d}(\xi,\eta)=0$ if $\xi=\eta$, and $\mathtt{d}(\xi,\eta)=e^{-|\xi\wedge\eta|}$ otherwise. Under this metric, it is clear that $\partial\T$ forms a compact metric space. We denote the space of Borel probability measures on $\partial\T$ by $\M_1(\partial\T)$.

For a given set of edge retention probabilities $\bp=\{p_e\}_{e\in\E}\subset (0,1]$, we consider the gauge function $\Psi_\bp:\V\to\R_+$ defined by
\begin{equation*}
    \Psi_{\bp}(x) := \frac{1}{\p_{\bp}[o\longleftrightarrow x]}.
\end{equation*}
We extend this function to the boundary $\partial \T$ by taking the limit $\Psi_{\bp}(\xi) = \lim_{n\to\infty} \Psi_{\bp}(x_n)$ for each ray $\xi=(x_0=o, x_1, \dots) \in \partial\T$.

A cornerstone of our proofs is a well-known result in discrete potential theory due to R. Lyons \cite{LP-Book}*{Theorem 16.3}, \cite{Lyons1992}*{Theorem 2.3}. It establishes that for any Borel set $E\subseteq\partial \T$, its connectivity probability is governed by its capacity:
\begin{equation}\label{connectivity cap relation}
    \cpy_{\bp}(E) \leq \p_{\bp}[o \longleftrightarrow E] \leq 2 \cpy_{\bp}(E).
\end{equation}
Here, $\cpy_{\bp}(E)$ denotes the capacity of $E$ with respect to the gauge function $\Psi_{\bp}$, defined as
\begin{equation*}
    \cpy_{\bp}(E) := \Big(\inf \big\{\e_\bp(\mu): \mu\in \M_1(\partial \T), \; \mu(\partial \T\setminus E)=0\big\} \Big)^{-1},
\end{equation*}
where $\e_\bp$ is the corresponding energy functional:
\begin{equation}\label{eq:energy_def}
    \e_\bp(\mu) = \iint\limits_{(\partial\T)^2}\Psi_{\bp}(\xi\wedge \eta)\mu(d\xi)\mu(d\eta).
\end{equation}

When the tree $\T$ is infinite, the energy $\e_\bp(\mu)$ may assume the value $+\infty$. We thus define the \emph{effective domain} of $\e_\bp$ as $\dom \e_{\bp}:=\{\mu\in\M_1(\partial\T): \e_{ \bp}(\mu)<\infty\}$. We conclude this section by noting that if $\T$ is infinite and possesses no finite rays, \eqref{connectivity cap relation} implies that percolation occurs under $\p_\bp$ if and only if $\dom\e_\bp\neq \emptyset$.

\section{Proof of Theorem \ref{main theorem}}\label{proof section}

In the proof of Theorem \ref{main theorem}, we assume that the underlying tree $\T$ has no finite rays. 
This does not restrict the generality because,  as $\T$ is locally finite, any infinite connected component must necessarily contain an infinite ray.

\subsection{Proof of part (i): stability in the absence of infinite clusters}\label{Stability in the absence of infinite clusters}

For $n\in\N$, let $\T_n=(\V_n,\E_n)$ be the truncation of $\T$ to the first $n$ levels, i.e., $\T_n$ is the finite subtree of $\T$ containing precisely the vertices at distance at most $n$ from the root.
Then the boundary $\partial\T_n$ may be identified with $V_n$.
Now consider the perturbed retention probabilities $\bq\bp=\{\min\{q_{|e|}p_e,1\}\}_{e\in\E}$ and the associated energy functional $\e^{\T_n}_{\bq\bp}$ in the truncation $\T_n$. Then
\begin{equation}\label{rel between perturbed and original energies}
    \e^{\T_n}_{\bq\bp}(\mu)
    =
    \sum_{v_1,v_2\in V_n} \frac{\mu(v_1)\mu(v_2)}{\p_{\bq \bp}[o\longleftrightarrow v_1\wedge v_2]}
    \geq
    \sum_{v_1,v_2\in V_n} \frac{\mu(v_1)\mu(v_2)}{\p_{\bp}[o\longleftrightarrow v_1\wedge v_2] \cdot Q_{|v_1\wedge v_2|}}\ge\frac{1}{Q_n}\e^{\T_n}_{\bp}(\mu)\,,
\end{equation}
where $\mu\in\M_1(V_n)$, $Q_0=1$, and for $k\in\N$, $Q_k:=\prod_{i=1}^k q_i$.
Note that here we used the assumption $q_i\ge 1$, which implies that $Q_n$ is non-decreasing, and furthermore, that $\p_{\bq\bp}[o\longleftrightarrow u)\ge\p_\bp[o\longleftrightarrow u]$ for all $u\in\V$. 
Then, by employing \eqref{connectivity cap relation}, one obtains that
\begin{equation}\label{relation between P_p(o and E_n) and P_{bar p}(o and E_n)}
    \p_{\bq\bp}[o\longleftrightarrow V_n]
    \leq
    2Q_n \p_{\bp}[o\longleftrightarrow V_n].
\end{equation}
The claim (i) of Theorem \ref{main theorem} now follows from (\ref{relation between P_p(o and E_n) and P_{bar p}(o and E_n)}) noting that 
\[\p_{\bq \bp}[o\longleftrightarrow \partial\T]=\inf\limits_{n\in\N}\p_{\bq \bp}[o\longleftrightarrow V_n]=\lim\limits_{n\to\infty}\p_{\bq \bp}[o\longleftrightarrow V_n]\,.\]

\subsection{Proof of part (ii): stability of infinite clusters under perturbations}\label{Stability of infinite clusters under perturbations}
We use the following elementary lemmas in the proof of \ref{main theorem} (ii).
\begin{lemma}\label{elementary lemma}
    Assume $\{a_n\}_{n\geq 0}$ is a non-negative summable sequence. 
        Then there exists a positive strictly increasing unbounded sequence $(\mathfrak{Q}_n)_{n\geq 0}$ such that $\lim_{n}\frac{\mathfrak{Q}_{n+1}}{\mathfrak{Q}_n}=1$ and $\sum_{n\geq 0} a_n \mathfrak{Q}_n<\infty$.
\end{lemma}
\begin{proof}
    It is clear that there exists a strictly increasing sequence $\{N_k\}_{k\in \N}$ of natural numbers such that for every $k\in\N$,
    \begin{equation*}
        \sum_{n=N_k+1}^{N_{k+1}} a_n<\frac{1}{2^k}.
    \end{equation*}
    We set, for $0\leq n\leq N_1$, $\mathfrak{\tilde Q}_n:=1$, and for every $k$ and $N_k<n\leq N_{k+1}$, set $\mathfrak{\tilde Q}_n:=k$.
    Then 
    \begin{eqnarray}
        \sum_{n=0}^\infty a_n \mathfrak{\tilde Q}_n
        &=&\notag 
        \sum_{n=0}^{N_1} a_n \mathfrak{\tilde Q}_n
        +
        \sum_{k=1}^\infty \sum_{n=N_k+1}^{N_{k+1}} a_n\mathfrak{\tilde Q}_n\\
        &\leq& \notag 
        \sum_{n=0}^{N_1} a_n
        +
        \sum_{k=1}^\infty
        \frac{k}{2^k}<\infty\,.
    \end{eqnarray}
Therefore, the sequence $\{\mathfrak{\tilde Q}_n\}_{n\geq }$, which is non-decreasing, satisfies the claims of the lemma. 
To construct a strictly increasing sequence, we set $\mathfrak{Q}_n:=\mathfrak{\tilde Q}_n (1-\frac{1}{n+2})$, \; $n\geq 0$.
\end{proof}
\begin{lemma}\label{lemma for proving necessary cond}   
Assume $\{a_n\}_{n\geq0}$ is a non-negative sequence and $\{b_n\}_{n\geq0}$ is a non-increasing sequence with $\lim_{n\to\infty}b_n=0$. 
If $\{|\sum_{i=0}^n a_ib_i-(\sum_{i=0}^n a_i)b_{n+1}|\}_{n\in\N}$ is bounded, then $(\sum_{i=0}^n a_i)b_{n+1}$ converges to $0$ as $n\longrightarrow\infty$.
\end{lemma}

\begin{proof}[Proof]
By the assumptions, there exists an $\mathfrak{M}>0$ such that for all $n\in\N$,
$$
\Big\lvert \sum_{i=0}^n a_ib_i-(\sum_{i=0}^n a_i)b_{n+1}\Big\rvert
=
\Big\lvert \sum_{i=0}^na_i(b_i-b_{n+1}) \Big\rvert
\leq \mathfrak{M}\,.
$$
Thus, since $\lim_{n\to\infty} b_n=0$, for any $m\in\N$,
\begin{equation*}
\mathfrak{M}
\geq
\lim_{n\to\infty}\sum_{i=0}^m a_i(b_i-b_{n+1})
=
\sum_{i=0}^m a_ib_i,   
\end{equation*}
which yields $\sum_{i=0}^\infty a_ib_i\leq \mathfrak{M}$.
Then the monotonicity of $\{b_n\}$ implies that for any $0\leq m<n$,
\begin{equation*}
    0\leq \sum_{i=m+1}^n a_ib_{n+1}\leq \sum_{i=m+1}^\infty a_ib_i\,.
\end{equation*}
Hence, since $\lim_{n\to\infty} b_n=0$, one has
\begin{equation*}
    \limsup_{n\to\infty}\sum_{i=0}^n a_i b_{n+1}
    \leq 
    \sum_{i=m+1}^\infty a_ib_i,
\end{equation*}
which, in turn, yields that 
\begin{equation*}
    \limsup_{n\to\infty}\sum_{i=0}^n a_i b_{n+1}
    \leq 
    \lim_{m\to\infty}\sum_{i=m+1}^\infty a_ib_i
    =
    0\,.
\end{equation*}
\end{proof}

\begin{proof}[Proof of Theorem \ref{main theorem}: Part (ii)]

Consider the following $\Psi_{\bp}-$level decomposition of the boundary $\partial\T$: 
\[\partial\T=\fr_\infty\cup\bigcup_{n=1}^\infty\fr_n\,,
\]
where $\fr_\infty=\{\xi\in\partial\T\,:\,\Psi_\bp(\xi)=\infty\}$, and $\fr_n=\{\xi\in\partial\T\,:\Psi_\bp(\xi)\le n\}$. 
One can easily verify that $\Psi_\bp:\partial\T\to [1,\infty]$ is lower semicontinuous. Therefore, for every $n\in\N$, $\fr_n\subset\partial\T$ is a compact subset. 
Since the decomposition is countable, we have
    \begin{equation}\label{uncountable cluster event decom}
    [o \xleftrightarrow[]{\hspace{0.2cm} u \hspace{0.2cm}} \partial\T]
    =
    [o \xleftrightarrow[]{\hspace{0.2cm} u \hspace{0.2cm}} \fr_\infty]
    \bigcup \bigcup_{n=1}^\infty
    [o \xleftrightarrow[]{\hspace{0.2cm} u \hspace{0.2cm}} \fr_n].
    \end{equation}
    Then, since $\p_\bp[o \xleftrightarrow[]{\hspace{0.2cm} u \hspace{0.2cm}} \partial\T]>0$, either $\p_\bp[o \xleftrightarrow[]{\hspace{0.2cm} u \hspace{0.2cm}} \fr_\infty]>0$ or there exists $n\in\N$ such that $\p_\bp[o \xleftrightarrow[]{\hspace{0.2cm} u \hspace{0.2cm}} \fr_n]>0$.
    Below, we discuss each case separately.

    \textit{\underline{"\textbf{Case I:} $\mathbf{\p_\bp[o \xleftrightarrow[]{\hspace{0.2cm} u \hspace{0.2cm}} \fr_\infty]>0}$"}:}   
    Firstly, to prove the almost sure existence of infinite clusters under $\p_{\bq\bp}$ (i.e., $\p_{\bq \bp}[o \longleftrightarrow \partial \T]>0$) for a given divergent perturbation $\bq=\{q_n\}_{n\in\N}$, it suffices to show 
    that $\p_{\bq \bp}[o \longleftrightarrow \fr_\infty]>0$, which, in light of \eqref{connectivity cap relation}, amounts to showing
    $\cpy_{\bq \bp}(\fr_\infty)>0$.
    By definition, the latter is equivalent to 
    \begin{equation}\label{condition to be met}
        \inf \Big\{\e_{\bq \bp}(\mu): \mu\in \M_1(\partial \T), \mu(\partial\T\setminus \fr_\infty)=0\Big\} <\infty.
    \end{equation}
    Secondly, assuming $\prod_{n\in\N} q_n=0$, Proposition 5.27 in \cite{LP-Book} implies that  $\p_{\bq\bp}[o \xleftrightarrow[]{\hspace{0.2cm} u \hspace{0.2cm}} \partial\T]=\p_{\bq\bp}[o \longleftrightarrow \partial\T]$. 
    Therefore, if we can show $\p_{\bq \bp}[o \longleftrightarrow \fr_\infty]>0$, which is equivalent to (\ref{condition to be met}), then this 
    automatically implies that $\p_{\bq\bp}[o \xleftrightarrow[]{\hspace{0.2cm} u \hspace{0.2cm}} \partial\T]>0$. 
    As a result, it suffices to verify \eqref{condition to be met} for some divergent perturbation $\bq=\{q_n\}_{n\in\N}$. We establish this in what follows.
    
    By our assumption, $\p_\bp[o \longleftrightarrow \fr_\infty]\ge \p_\bp[o \xleftrightarrow[]{\hspace{0.2cm} u \hspace{0.2cm}} \fr_\infty]>0$ and thus by \eqref{connectivity cap relation}, we have $\cpy_{\bp}(\fr_\infty)>0$. 
    Whence, there exists $\tilde \mu\in \M_1(\partial \T)$ such that $\tilde \mu(\partial\T\setminus \fr_\infty)=0$ and  $\mathcal{E}_{\bp}(\tilde\mu)<\infty$. Now
    \begin{equation}\label{the int at criticality-0}
        \mathcal{E}_{\bp}(\tilde\mu)
        =
        \iint\limits_{(\partial\T)^2}\Psi_{\bp}(\xi\wedge \eta)\tilde\mu(d\xi)\tilde\mu(d\eta)
        = 
        \iint\limits_{\Delta} \Psi_\bp(\xi\wedge\eta) \tilde\mu(d\xi)\tilde\mu(d\eta) 
        +
        \sum_{v\in\V} \frac{\tilde\mu^2(\Theta_v)}{\p_{\bp}[o\longleftrightarrow v]},
    \end{equation}
    where for $v\in\V$,
    $$
    \Theta_v:=\{(\xi,\eta)\in (\partial \T)^2: \xi\wedge\eta=v\}
    \; \text{ and } \;
    \Delta:=\{(\xi,\xi): \xi\in\partial\T\}=\mathtt{diag}(\partial\T)^2.
    $$
    Note that the second identity in \eqref{the int at criticality-0} holds because $\{\Delta, \Theta_v, \; v\in \V\}$ forms a measurable partition of $(\partial \T)^2$.
    As $\tilde\mu$ is supported on $\fr_\infty$,
    \begin{equation*}
        \iint\limits_{\Delta}\Psi_{\bp}(\xi\wedge \eta)\tilde\mu(d\xi)\tilde\mu(d\eta)
        =
        \iint\limits_{\Delta(\fr_\infty)}\Psi_{\bp}(\xi\wedge \eta)\tilde\mu(d\xi)\tilde\mu(d\eta),
    \end{equation*}
    where $\Delta(\fr_\infty):=\mathtt{diag}\; \fr_\infty^2$.
    Thus, since $\Psi_\bp(\xi)=\infty$ for all $\xi\in\fr_\infty$, one obtains from \eqref{the int at criticality-0} that 
    \begin{equation}\label{the int at criticality}
         \infty>\mathcal{E}_{\bp}(\tilde\mu)
        =
        \infty\cdot \tilde\mu^2(\Delta)
        +
        \sum_{v\in\V} \frac{\tilde\mu^2(\Theta_v)}{\p_{\bp}[o\longleftrightarrow v]}\,.
    \end{equation}
    Thus, in particular, $\tilde\mu^2(\Delta)=0$. Therefore,
    \begin{equation*}
        \mathcal{E}_{\bp}(\tilde\mu)
        =
        \sum_{v\in\V} \frac{1}{\p_{\bp}[o\longleftrightarrow v]}\tilde\mu^2(\Theta_v)
        =
        \sum_{n=0}^\infty \sum_{\substack{v\in\V\\ |v|=n}}
        \frac{1}{\p_{\bp}[o\longleftrightarrow v]}\tilde\mu^2(\Theta_v)
        <\infty.
    \end{equation*}
    Applying Lemma 
    \ref{elementary lemma} to the sequence 
    $
    a_n:=\sum_{\substack{v\in\V,\, |v|=n}}
    \frac{1}{\p_{\bp}[o\longleftrightarrow v]}\tilde\mu^2(\Theta_v)
    $
    yields a strictly increasing sequence $\{\tilde Q_n^{-1}\}_{n\geq 0}$ such that $\tilde Q_n^{-1}\xrightarrow[n\to\infty]{} \infty$, $\frac{\tilde Q_n}{\tilde Q_{n-1}}\xrightarrow[n\to\infty]{} 1$, and 
    \begin{equation*}
        \sum_{n=0}^\infty 
        \tilde Q_n^{-1}
        \sum_{\substack{v\in\V\\ |v|=n}}
        \frac{1}{\p_{\bp}[o\longleftrightarrow v]}\tilde\mu^2(\Theta_v)
        <
        \infty.
    \end{equation*}
    It is readily checked that the perturbation $\tilde\bq=\{\tilde q_n\}_{n\in\N}$ given by $\tilde q_1:=\tilde Q_1$ and $\tilde q_n:=\frac{\tilde Q_{n}}{\tilde Q_{n-1}}$, $n\geq 2$ 
    satisfies $\mathcal{E}_{\tilde \bq \bp}(\tilde\mu)<\infty$, and whence also \eqref{condition to be met} holds.
    Indeed, for the perturbed retention probabilities $\tilde\bq \bp=\{\tilde q_{|e|} p_e\}_{e\in\E}$ and the corresponding gauge function $\Psi_{\tilde\bq \bp}$ one has 
    \begin{equation*}
        \Psi_{\tilde\bq \bp}(x)
        =
        \frac{1}{\p_{\tilde\bq \bp}[o\longleftrightarrow x]}
        =
        \Psi_{\bp}(x)\prod_{i=1}^{|x|}\tilde q_i^{-1}
        =
        \Psi_{\bp}(x) \tilde Q_{|x|}^{-1}, \;\;\; x\in \V.
    \end{equation*}
    Thus, as $\tilde\mu^2(\Delta)=0$,
    \begin{equation*}
        \mathcal{E}_{\tilde\bq \bp}(\tilde\mu)
        =
        \iint\limits_{(\partial\T)^2}\Psi_{\tilde\bq \bp}(\xi\wedge \eta)\tilde\mu(d\xi)\tilde\mu(d\eta)
        =
        \sum_{n=0}^\infty 
        \tilde Q_n^{-1}
        \sum_{\substack{v\in\V\\ |v|=n}}
        \frac{1}{\p_{\bp}[o\longleftrightarrow v]}\tilde\mu^2(\Theta_v)
        <
        \infty\,.
    \end{equation*}
    
\textit{\underline{"\textbf{Case II:} $\mathbf{\p_\bp[o \xleftrightarrow[]{\hspace{0.2cm} u \hspace{0.2cm}} \fr_\infty]=0}$"}:} 
In this case, there exists $n\in\N$ such that $\p_\bp[o \xleftrightarrow[]{\hspace{0.2cm} u \hspace{0.2cm}} \fr_n]>0$.
This, in particular, implies that the set $\fr_n\subset \partial\T$ is uncountable.
Since $\Psi_\bp(\xi)\leq n$ for all $\xi\in\fr_n$, for every $\mu\in\M_1(\partial\T)$ with $\mu(\partial\T\setminus\fr_n)=0$, one has that 
\begin{equation*}
\e_\bp(\mu)=\iint_{(\partial\T)^2}\Psi_\bp(\xi\wedge\eta)\mu(d\xi)\mu(d\eta)\leq n\,.    
\end{equation*}
Then, as the set $\fr_n$ is compact and uncountable, there exists a \textit{non-atomic} probability measure $\tilde \mu\in\M_1(\partial\T)$ such that  $\tilde\mu(\partial\T\setminus\fr_n)=0$ (see e.g. Theorem 12.22 in \cite{AB-book}).
Hence, in particular, $\tilde\mu^2(\Delta)=0$.
Then, with a similar reasoning as in Case I, we may construct a strictly increasing sequence $\{\tilde Q_n^{-1}\}_{n\geq 0}$ and a perturbation $\tilde\bq=\{\tilde q_n\}_{n\in\N}$ such that $\prod_{n\in\N}\tilde q_n=0$ 
and so that, for the perturbed retention probabilities $\tilde\bq \bp:=\{\tilde q_{|e|} p_e\}_{e\in\E}$, we have $\e_{\tilde\bq \bp}(\tilde\mu)<\infty$.
In light of \eqref{connectivity cap relation}, this then implies $\p_{\tilde\bq \bp}[o \longleftrightarrow \fr_n]>0$.

    Finally, we prove the necessity of the condition (\ref{eq: a necessary condition for the persistence in infinity clusters case}), when the percolation model $(\T, \bp=\{p\}_{e\in\E})$ is homogeneous.
    For every integer $n\geq 0$, we define 
    \begin{equation}\label{definition of Chi_n and it is union as square sets}
        \Xi_n:=\{(\xi,\eta)\in \partial \T^2: |\xi\wedge \eta|\geq n \}=\bigsqcup\limits_{v\in \V: |v|=n}\partial \T_v^2,
    \end{equation}
    where $\partial \T_v$ is the set of infinite rays passing through the vertex $v$.
    For the perturbed retention probabilities $\bq\bp=\{p_n\}_{n\in\N}=\{q_np\}_{n\in\N}$, we then consider the corresponding energy functional $\e_{\bq \bp}$: for $\mu\in\M_1(\partial\T)$,
    $$
    \e_{\bq \bp}(\mu)
        =
        \iint\limits_{(\partial\T)^2}\Psi_{\bq\bp}(\xi\wedge \eta)\mu(d\xi)\mu(d\eta)
        =
        \infty \cdot  \mu^2(\Delta)
        +
        \sum_{n=0}^\infty (p_1 \dots p_n)^{-1} \mu^2(\Xi_n\setminus\Xi_{n+1}).
    $$
    We now aim to show that for any $\mu\in \dom \e_{\bq \bp}$,
    \begin{equation}\label{eq to show to prove nec cond}
        \e_{\bq \bp}(\mu)
        =
        1+\sum_{n=0}^\infty \left(p^{-n-1} Q^{-1}_{n+1}-p^{-n} Q^{-1}_{n}\right)\mu^2(\Xi_{n+1})\,,
    \end{equation}
    where $Q_0=1$ and $Q_n=\prod_{i=1}^n q_i$ for $n\ge 1$.
    For a fixed $N\in \N$, let
    \[\e(N)=\mu^2(\partial\T\setminus\Xi_1)+\sum_{n=1}^{N} (p_1\dots p_n)^{-1}\mu^2(\Xi_n\setminus\Xi_{n+1})\,,\]
    and note that $\e_{\bq \bp}(\mu)=\lim_{N\to\infty}\e(N)$. Writing 
\begin{equation}\label{eq: e^N_{bar p}(mu) and F_n's and Q_n's}
        \e(N)=
        1-p^{-N}Q_N^{-1}\mu^2(\Xi_{N+1})
        +
        \sum_{n=0}^{N-1}
        [p^{-n-1}Q^{-1}_{n+1}-p^{-n}Q^{-1}_{n}]\mu^2(\Xi_{n+1})
\end{equation}
and noting that $\mu^2(\Xi_n)\downarrow \mu^2(\Delta)=0$ as $n\to\infty$, we may apply Lemma \ref{lemma for proving necessary cond} with $a_n=p^{-n-1}Q^{-1}_{n+1}-p^{-n}Q^{-1}_{n}$ and $b_n=\mu^2(\Xi_{n+1})$ to conclude that $\lim_{N\to\infty}p^{-N}Q_N^{-1}\mu^2(\Xi_{N+1})=0$. Thus, \eqref{eq to show to prove nec cond} follows from \eqref{eq: e^N_{bar p}(mu) and F_n's and Q_n's}  upon taking the limit as $N\to\infty$.
\\
Since $\{\partial \T_v: |v|=n\}$ partitions the space $\partial\T$, by (\ref{definition of Chi_n and it is union as square sets}) and the Cauchy-Schwarz inequality,
\begin{equation*}
    \mu^2(\Xi_n)
    =
    \sum_{v: |v|=n}\mu(\partial \T_v)^2
    \geq 
    (\#\{v\in\V: |v|=n\})^{-1}\Big(\sum_{v: |v|=n}\mu(\partial \T_v)\Big)^2
    =
    \frac{1}{\#E_n}.
\end{equation*}
This, in combination with (\ref{eq to show to prove nec cond}), concludes the proof, i.e.,  for $\mu\in\dom \e_{\bq \bp}$,
$$
\infty
>
\e_{\bq \bp}(\mu)-1
\geq 
\sum_{n=0}^\infty [1-pq_{n+1}] p^{-n-1} Q^{-1}_{n+1} \frac{1}{\#E_{n+1}}\,.
$$

\end{proof}

\section{On a connection to the Erd\H{o}s similarity conjecture for Cantor sets}\label{Section for Applications}

In this section, we aim to prove Corollary \ref{cor: application of main thm to ESC}.
Let us consider the dyadic tree $\T=\bigcup_{n=0}^\infty\{0,1\}^n$, and let $\pi\colon\partial\T\to[0,1]$ denote the projection $(\varnothing,x_1,x_2\ldots)\mapsto\sum_{n\in\N}x_n2^{-n}$. Given a compact set $K\subset[0,1]$, let 
\[\T_K=\bigcup_{n=0}^\infty\left\{(\varnothing,x_1,\ldots,x_n)\,:\,\sum_{n\in\mathbb{N}}x_n 2^{-n}\in K\right\}\subset\T\] denote the subtree coding $K$ in the dyadic base. 
Note that for each subtree $\T'\subset\T$ without finite rays, letting $K=\pi(\partial\T')$, we have $\T'=\T_K$.
Given a sequence of weights $\{q_n\}_{n\in\N}\subset(0,1)$, let $\T_A(\omega)$ denote the connected component of the root in the $\bq=\{q_{|e|}\}_{e\in\E}$-percolated tree, from which all finite rays have been removed. By definition, $A=A(\omega)=\pi(\partial \T_A(\omega))$ is then the random fractal percolation set obtained using the level-dependent percolation weights $\{q_n\}$ on $\T$.

\begin{proposition}\label{Erdos application}
Let $C\subset[0,1]$ be a Cantor set. There is a constant $c>0$ and a sequence $\{q_n\}_{n\in\N}\subset(0,1)$ (allowed to depend on $C$) with $\prod_{n\in\N} q_n=0$, such that for all $\tfrac14\le\lambda\le\tfrac12$, $0<t<\tfrac12$,
\begin{equation}\label{main theorem applied to a fixed copy}
    \p_{\bq}[o\longleftrightarrow \partial\T_{\lambda C+t}]>c\,.
\end{equation} 
\end{proposition}

\begin{proof}
Applying the second part of Theorem \ref{main theorem} for the tree $\T_C$ and the percolation weights $p_e\equiv 1$, it follows that for some $\{q_n\}_{n\in\N}$ with $\prod_{n\in\N} q_n=0$, 
\begin{equation}\label{perc C fixed}
    \p_{\bq}[o\longleftrightarrow \partial\T_{C}]>0\,. 
\end{equation}

For a measure $\bar\nu$ on $\R$, define
\[\bar\e_\bq(\bar\nu):=\int\int\bar\Psi(x,y)\,d\bar\nu(x)\,d\bar\nu(y)\,,\]
where $\bar\Psi(x,y)=\prod\limits_{1\leq i\leq -\log_2|x-y|}
q_i^{-1}$. It is then a standard procedure to check (see e.g. \cite{PemantlePeres1996}*{Section 3.1}) that for some constant $1\le C=C_{\bq}<\infty$, 
\begin{equation}\label{from dyadic to euclidean}
    \frac{\e_\bq(\nu)}{C}\le\bar\e_\bq(\nu\circ\pi^{-1})\le C\e_\bq(\nu)\,,
\end{equation}
whenever $\nu$ is a measure on $\partial\T$ (recall \eqref{eq:energy_def}).
Moreover, since $\bar{\Psi}$ is translation invariant and for all  $\lambda>\tfrac14$, $\bar\Psi(\lambda x,\lambda y)\le (\inf_{n\in\N}{q_n})^{-2}\bar\Psi(x,y)$, it follows that for all $\tfrac14<\lambda<\tfrac12$, $0<t<\tfrac12$, 
\begin{equation}\label{comparable energies}
    \frac{\e_\bq(\nu)}{C}\le\bar\e_\bq(\nu\circ\pi^{-1}\circ h_{\lambda,t})\le C'\e_\bq(\nu)\,,
\end{equation}
where $h_{\lambda, t}(x)=(x-t)/\lambda$ and $C'=(\inf_\N q_n)^{-2}\cdot C$. 
Let $\mu$ be a measure supported on $\partial\T_C$ such that $\e_\bq(\mu)<\infty$.
Then $\mu_{\lambda,t}:=\mu\circ\pi^{-1}\circ h_{\lambda,t}\circ\pi$ is a measure supported on $\T_{\lambda C+t}$ and combining \eqref{from dyadic to euclidean} and \eqref{comparable energies},
$\frac{\e_\bq(\mu)}{C''}\le\e_\bq(\mu_{\lambda,t})\le C''\e_\bq(\mu)$ for some finite constant $C''$. The claim is now an immediate consequence of \eqref{connectivity cap relation}.
\end{proof}

\begin{proof}[Proof of Corollary \ref{cor: application of main thm to ESC}]

Let $\{q_n\}_{n\in\N}$ be given by the previous proposition, and let $A=A(\omega)$ be the corresponding random fractal percolation set. Since $\prod_{n=1}^\infty q_n=0$, it follows that $\p_\bq(z\in A(\omega))=0$ for all $z\in\R$. Hence, by Fubini's theorem, $A(\omega)$ is Lebesgue-null for $\p_\bq$-almost all $\omega$. Now let $A^{(n)}=A^{(n)}(\omega_n)$ be independent copies of the random set $A$, and set $B=B(\overline\omega):=\bigcup_{n=1}^\infty A^{(n)}(\omega_n)$, where $\overline\omega=(\omega_1,\omega_2,\ldots)$ and the product space $\Omega^\N$ is endowed with the probability measure $\p_\bq^\N$.
Noting that $o\longleftrightarrow\partial\T_{\lambda C+t}$ and $A\cap(\lambda C+t)\neq\varnothing$ are the same event, and using \eqref{main theorem applied to a fixed copy} and the independence of the $A^{(n)}$, we have
\begin{equation}\label{eq: property of B, Borel-Cantelli}
    \p_{ \bq}^\N\left(B(\overline\omega)\cap(\lambda C+t)=\varnothing\right)=0
\end{equation}
for each $(\lambda, t)\in \mathcal{R}:=[\frac{1}{4},\frac{1}{2}]\times (0,\frac{1}{2})$.
Defining $\mathcal{H}(\overline\omega):=\{(\lambda,t)\in \mathcal{R}: B(\overline\omega)\cap(\lambda C+t)=\varnothing \}$, Fubini's theorem implies that
\begin{equation}
    0=\int_\mathcal{R}\int_{\Omega^\N} \mathds{1}_{\mathcal{H}(\overline\omega)}(\lambda, t) \p_\bq^\N(d\overline\omega)\,d\lambda\, dt 
    =\int_{\Omega^\N}\La^2(\mathcal{H}(\overline\omega)) \p_\bq^{\N}(d\overline\omega)\,,
\end{equation}
where $\La^2$ denotes the Lebesgue measure on the Euclidean plane.
Consequently, $\La^2(\mathcal{H}(\overline\omega))=0$ for almost every $\overline\omega$. Noting also that $\La(B(\overline\omega))=0$ almost surely by sigma-additivity of measure and the fact that $A(\omega)$ is Lebesgue-null for $\p_\bq$-almost all $\omega$, we may pick $\overline\omega$ such that $0=\La^2(\mathcal{H}(\overline\omega))=\La(B(\overline\omega))$. Then $B(\overline\omega)\cap(\lambda C+t)\neq\varnothing$ for almost all $(\lambda,t)\in\mathcal{R}$. Finally, for  $\widetilde B:=\bigcup_{a,b\in\Q\setminus\{0\}} \left(a B(\overline\omega) +b\right)$, it follows that $\La(\widetilde{B})=0$ and $\widetilde B\cap(\lambda C+t)\neq\varnothing$ for almost all $\lambda\neq 0$, $t\in\R$.
\end{proof}
\begin{remark}\label{rmk: about the Erdos similarity conjecture}
\begin{enumerate}
    \item[(1)] Kolountzakis has obtained related results for an arbitrary infinite set $C$.
    In \cite{Kou1997}, it is shown that there are sets $B$ of arbitrarily small (but positive) Lebesgue measure such that $(\lambda C+t)\cap B\neq\varnothing$ for Lebesgue-almost all $(\lambda,t)$. The complement of the set $B$ is a modified fractal percolation set on $[0,1]$ with $q_n$ tending to one so fast that the Lebesgue measure of the percolation set is positive. Using a deterministic construction, Koulontzakis provides a set $B$ of arbitrarily small positive Lebesgue measure such that, for almost every $\lambda$, $(\lambda C+t)\cap B\neq\varnothing$, holds for all $t\in\R$. 
    If $C$ is a Cantor set, Shmerkin and Yavicoli \cite{ShmerkinYavicoli2025} have shown that there is $B\subset\R$ of zero Lebesgue-measure such that for all bi-Lipschitz functions $f\colon\R\to\R$, $(f(C)+t)\cap B\neq\varnothing$ for almost every $t\in\R$. 
\item[(2)] Homogeneous fractal percolation sets $A$, i.e. the ones constructed using a constant sequence $q_n=q$, also bear some interest for the Erd\H{o}s similarity problem. Namely, they may be used to show that each set of positive Hausdorff dimension satisfies the conjecture, as shown by Jun, Lai, and Mooroogen \cite{JLM2025}*{Corollary 3.8} using results from \cite{ShmerkinSuomala2018}. The
proof of this fact rests on two observations that are valid if $q<1$ is chosen large enough depending on the dimension of $C$:\\
\emph{$\alpha$)} For a fixed parameter $(\lambda,t)$, the intersection $(\lambda C+t)\cap A(\omega)$ is 'large' with high probability, where largeness is measured in some appropriate way.\\ 
\emph{$\beta$)} For almost every realisation of the fractal percolation $A(\omega)$, the 'largeness' of $(\lambda C+t)\cap A(\omega)$ is a continuous function of $(\lambda,t)$.\\
In \cite{JLM2025} and \cite{ShmerkinSuomala2018}, a quantitative size condition (positive Hausdorff dimension) for $C$ was used to verify both properties $\alpha$) and $\beta$). 
The Proposition \ref{Erdos application} strengthens this approach by showing that the Property $\alpha$) is valid for all $\tfrac14\le\lambda\le\tfrac12$, $0<t<\tfrac12$ under the minimal assumption that $C$ is a Cantor set. Here, 'largeness' is measured in terms of $\cpy_{\bq}(\partial{\T_{\lambda C+t}})$. It remains to be seen if $\beta$) could also be shown to hold under the same assumption, leading to the resolution of the Erd\H{o}s similarity problem for Cantor sets.
\item[(3)] Theorem \ref{main theorem} (ii) concerns the stability of percolation clusters under perturbations $q_e$ of the probability weights. 
The validity of $\beta$) in the setting of Proposition \ref{Erdos application} may also be considered a \emph{stability of percolation} type problem, but now the perturbations are induced by scaling and translating the set $C$: the tree $\T_C$ gets replaced by $\T_{\lambda C+t}$, where $\lambda$ and $t$ vary continuously.
\end{enumerate}
\end{remark}

\section*{Statement on the use of AI}
AI tools were used for limited language editing only and played no role in the mathematical work or any other aspect of the paper's preparation.

\section*{Acknowledgments}
We thank Jeff Steif for many valuable suggestions on an early draft of the manuscript.
We are also grateful to the anonymous referees for their helpful comments to improve the presentation of the paper.

\end{document}